\documentclass[11pt]{amsart}
\usepackage[margin=1in]{geometry}
\usepackage{latexsym}
\usepackage{amsfonts}
\usepackage{amsmath}
\usepackage{amssymb}
\usepackage{amsthm}
\usepackage{enumerate}
\usepackage[hang,flushmargin]{footmisc}
\usepackage{caption}
\usepackage{tabu}
\usepackage{mathrsfs}
\usepackage{subfig}
\usepackage[foot]{amsaddr}
\usepackage{float}

\usepackage{graphicx}
\usepackage{epstopdf}
\usepackage{epsfig}
\usepackage{caption}
 
\usepackage{bm} 
\usepackage{cite}

\newtheorem{theorem}{Theorem}[section]
\newtheorem{lemma}{Lemma}[section]
\newtheorem{proposition}{Proposition}[section]
\newtheorem{corollary}{Corollary}[section]

\newtheorem{conjecture}{Conjecture}[section]

\theoremstyle{definition}
\newtheorem{definition}{Definition}[section]
\newtheorem{remark}{Remark}[section]

\DeclareMathOperator{\IR}{IR}
\DeclareMathOperator{\Lon}{Lon}
\DeclareMathOperator{\Soc}{Soc}
\DeclareMathOperator{\pn}{pn}
\DeclareMathOperator{\coll}{coll}
\DeclareMathOperator{\fold}{fold}

\begin{document}
\title[Isoperimetry, Stability, and Irredundance in Direct Products]{Isoperimetry, Stability, and Irredundance in Direct Products}

\author{Noga Alon$^*$}
\address{$^*$Princeton University and Tel Aviv University}
\email{nalon@math.princeton.edu}

\author{Colin Defant$^\dagger$}
\address{$^\dagger$Princeton University}
\email{cdefant@princeton.edu}

\begin{abstract}
The direct product of graphs $G_1,\ldots,G_n$ is the graph with vertex set $V(G_1)\times\cdots\times V(G_n)$ in which two vertices $(g_1,\ldots,g_n)$ and $(g_1',\ldots,g_n')$ are adjacent if and only if $g_i$ is adjacent to $g_i'$ in $G_i$ for all $i$. Building off of the recent work of Brakensiek, we prove an optimal vertex isoperimetric inequality for direct products of complete multipartite graphs. Applying this inequality, we derive a stability result for independent sets in direct products of balanced complete multipartite graphs, showing that every large independent set must be close to the maximal independent set determined by setting one of the coordinates to be constant. Armed with these isoperimetry and stability results, we prove that the upper irredundance number of a direct product of balanced complete multipartite graphs is equal to its independence number in all but at most $37$ cases. This proves most of a conjecture of Burcroff that arose as a strengthening of a conjecture of the second author and Iyer. We also propose a further strengthening of Burcroff's conjecture.   
\end{abstract}

\maketitle

\bigskip

\section{Introduction}\label{Intro}

All graphs in this paper are assumed to be simple. We denote the vertex set and edge set of a graph $G$ by $V(G)$ and $E(G)$, respectively. The letter $\mu$ will denote the uniform probability measure on $V(G)$. That is, $\mu(S)=|S|/|V(G)|$ for all $S\subseteq V(G)$. The \emph{direct product} (also called the \emph{tensor product}, \emph{Kronecker product}, \emph{weak product}, or \emph{conjunction}) of graphs $G_1,\ldots,G_n$, denoted by either $G_1\times\cdots\times G_n$ or $\prod_{i=1}^nG_i$, is the graph with vertex set $V(G_1\times\cdots\times G_n)=V(G_1)\times\cdots\times V(G_n)$ in which two vertices $(g_1,\ldots,g_n)$ and $(g_1',\ldots,g_n')$ are adjacent if and only if $\{g_i,g_i'\}\in E(G)$ for all $i\in[n]$. We also use the product notation $\prod_i G_i$ to denote a direct product of a collection of graphs. Much of this paper is devoted to studying direct products of balanced complete multipartite graphs, which are complete multipartite graphs in which the partite sets all have the same size. More precisely, if $K[u,t]$ denotes the complete multipartite graph consisting of $t$ partite sets of size $u$, then we are concerned with graphs of the form $\prod_{i=1}^n K[u_i,t_i]$. 

One motivation for studying these graphs comes from the investigation of \emph{unitary Cayley graphs}, which are specific graphs associated to commutative rings with unity. Unitary Cayley graphs have become a popular topic over the past few decades \cite{burcroff, coolnames, defant, defant2, defantandiyer, dejter, fuchs, klotzsander, mandm, mandmclique}, in part because of their connection with a theorem of Erd\H{o}s and Evans \cite{eande} that led to the notion of the \emph{representation number} of a graph \cite{Akhtar, Akhtar2, Akhtar3, Evans, Evans2, Gallian, Narayan} (see Section 7.6 of \cite{Gallian} for more details). The authors of \cite{coolnames} have used a structure theorem for Artinian rings to show that the unitary Cayley graph of a finite ring is isomorphic to a direct product of balanced complete multipartite graphs. 

Hundreds of papers in graph theory have studied what is called the \emph{domination chain}; this is a collection of graph parameters that always satisfy a certain chain of inequalities. The aim is usually to show that these inequalities are actually equalities for certain types of graphs. We only discuss three of these graph parameters and refer the interested reader to Section 3.5 of \cite{Haynes} for more information about the domination chain. The first parameter we consider is the \emph{independence number} of a graph $G$, denoted $\alpha(G)$, which is the largest size of an independent set in $G$. A related notion is that of the \emph{independence ratio} of a graph, which is defined by $\beta(G)=\alpha(G)/|V(G)|$. The \emph{closed neighborhood} of a set $S\subseteq V(G)$, denoted $N[S]$, is the union of $S$ with all of the neighbors of the vertices in $S$. We say $S$ is \emph{dominating} if $N[S]=V(G)$. We say $S$ is \emph{irredundant} if $N[S\setminus\{v\}]\neq N[S]$ for all $v\in S$. The \emph{upper domination number} of $G$, denoted $\Gamma(G)$, is the maximum size of an irredundant dominating set in $G$. The \emph{upper irredundance number} of $G$, denoted $\text{IR}(G)$, is the maximum size of an irredundant set in $G$. Every maximal independent set is an irredundant dominating set, and every irredundant dominating set is obviously an irredundant set. Therefore, we always have the chain of inequalities \[\alpha(G)\leq \Gamma(G)\leq\IR(G),\] which comprises the upper portion of the domination chain. One of the notable results concerning these parameters is a theorem of Cheston and Fricke, which shows that $\alpha(G)=\IR(G)$ whenever $G$ is strongly perfect \cite{Cheston}. 
 
Suppose now that $G=\prod_{i=1}^nK[u_i,t_i]$ is a direct product of balanced complete mulipartite graphs, where $t_1\geq\cdots\geq t_n$. It is straightforward to check that $\alpha(G)=|V(G)|/t_n$ (alternatively, $\beta(G)=1/t_n$). While studying domination parameters of unitary Cayley graphs, the second author and Iyer were led to conjecture that for these graphs $\alpha(G)=\Gamma(G)$ \cite{defantandiyer}. They proved this conjecture in the cases $t_n\leq 2$ and $n\leq 3$. Burcroff observed that none of the arguments proving those cases of the conjecture used the fact that the sets under consideration were dominating \cite{burcroff}. In other words, $\alpha(G)=\IR(G)$ when $t_n\leq 2$ or $n\leq 3$. She then made the following stronger conjecture. 

\begin{conjecture}[\!\!\cite{burcroff}]\label{Conj1}
If $G=\prod_{i=1}^nK[u_i,t_i]$ is a direct product of balanced complete multipartite graphs, then $\alpha(G)=\IR(G)$. 
\end{conjecture}

Making progress toward this conjecture, Burcroff proved the following theorem.

\begin{theorem}[\!\!\cite{burcroff}]\label{Thm1}
If $G=\prod_{i=1}^nK[u_i,t_i]$, where $t_1\geq\cdots\geq t_n\geq 2$, then \[\IR(G)\leq \min\left\{\alpha(G)+2t_2\cdots t_n,\,\frac{t_n^2}{2t_n-1}\alpha(G)\right\}.\]
\end{theorem}

In this article, we prove most of Conjecture \ref{Conj1}. More precisely, we explicitly list $37$ graphs $Z_1,\ldots,Z_{37}$ in Section \ref{Proof} and prove the following theorem. 

\begin{theorem}\label{Thm7}
Let $G=\prod_{i=1}^nK[u_i,t_i]$ be a direct product of balanced complete multipartite graphs. If $G$ is not one of the graphs $Z_1,\ldots,Z_{37}$ listed in Section \ref{Proof}, then \[\alpha(G)=\IR(G).\]  
\end{theorem} 
The proof of this theorem requires three main ingredients that are interesting in their own right. For the first ingredient, we consider the even more general family of graphs that can be written as a direct product of (not necessarily balanced) complete multipartite graphs. In Section \ref{Near}, we prove the following theorem via a simple application of the polynomial method. Observe that this theorem both strengthens and generalizes Theorem \ref{Thm1}. 

\begin{theorem}\label{Thm3}
Let $G=\prod_{i=1}^n H_i$, where each graph $H_i$ is a complete multipartite graph. If $S\subseteq V(G)$ is an irredundant set, then there exist sets $\Lon(S),\Soc(S)\subseteq S$ such that $\Lon(S)\cap\Soc(S)=\emptyset$, $\Lon(S)\cup\Soc(S)=S$, $\Lon(S)$ is an independent set in $G$, and $|\Soc(S)|\leq 2^n$. In particular, $\IR(G)\leq\alpha(G)+2^n$. 
\end{theorem} 

The second ingredient in the proof of Theorem \ref{Thm7} involves determining an optimal isoperimetric inequality for direct products of complete multipartite graphs. Isoperimetric inequalities are ubiquitous in extremal combinatorics and graph theory \cite{Alon, Bobkov, Bollobas, Bollobas2, Bollobas3, Bollobas4, Carlson, Christofides, Chung, Harper, Harper2, Leader, Wang}. For every graph $G$ and every set $S\subseteq V(G)$, the \emph{vertex boundary} $\partial S$ is defined by \[\partial S=\{w\in V(G):\{v,w\}\in E(G)\text{ for some }v\in S\}.\] Note that $\partial S$ can include elements of $S$ itself, but it is also possible to have elements of $S$ that are not in $\partial S$. The \emph{vertex isoperimetric profile} of a graph $G$ with respect to a measure $\tau$ on $V(G)$ is the function $\Phi_\tau(G,\cdot):[0,1]\to[0,1]$ defined by \[\Phi_\tau(G,\nu)=\min\{\tau(\partial S):S\subseteq V(G), \tau(S)\geq \nu\}.\] If we do not specify the measure $\tau$, then we assume $\tau$ is the uniform measure $\mu$ by default. That is, $\Phi(G,\nu):=\Phi_\mu(G,\nu)$.

Brakensiek essentially gave a recursive formula for $\Phi(G,\nu)$ in the case where $G$ is a direct product of complete graphs that all have the same size \cite{Brakensiek}. It turns out that his proof method generalizes substantially. Our proof of the following theorem, given in Section \ref{Iso}, closely follows Brakensiek's argument, which comprises Appendix B of \cite{Brakensiek}. 

\begin{theorem}\label{Thm2}
Let $H_1,\ldots,H_n$ be complete multipartite graphs such that $\beta(H_1)\leq\cdots\leq\beta(H_n)$ and \[\prod_{k\in A}\frac{1-\beta(H_k)}{\beta(H_k)}\geq \frac{1-\beta(H_n)}{\beta(H_n)}\] for all nonempty $A\subseteq [n-1]$. We have \[\Phi(H_1,\nu)=\begin{cases} 0, & \mbox{if } \nu=0; \\ 1-\beta(H_1), & \mbox{if } 0<\nu\leq\beta(H_1); \\ 1, & \mbox{if } \beta(H_1)<\nu\leq 1. \end{cases}\] If $n\geq 2$, then \[\Phi(H_1\times\cdots\times H_n,\nu)=\begin{cases} 0, & \mbox{if } \nu=0; \\ (1-\beta(H_n))\Phi\left(H_1\times\cdots\times H_{n-1},\dfrac{\nu}{\beta(H_n)}\right), & \mbox{if } 0<\nu\leq\beta(H_n); \\ 1-\beta(H_n)+\beta(H_n)\Phi\left(H_1\times\cdots\times H_{n-1},\dfrac{\nu-\beta(H_n)}{1-\beta(H_n)}\right), & \mbox{if } \beta(H_n)<\nu\leq 1. \end{cases}\] 
\end{theorem}

\begin{remark}\label{Rem1}
The hypothesis in Theorem \ref{Thm2} that $\displaystyle\prod_{k\in A}\frac{1-\beta(H_k)}{\beta(H_k)}\geq \frac{1-\beta(H_n)}{\beta(H_n)}$ for all nonempty $A\subseteq [n-1]$ is not a huge restriction. For example, this condition is satisfied if $\beta(H_i)\leq 1/2$ for all $i\in[n-2]$. In particular, it holds whenever the complete multipartite graphs $H_i$ are balanced.  
\end{remark}

We also prove the following useful corollary in Section \ref{Iso}. 
\begin{corollary}\label{Cor1}
If $H_1,\ldots,H_n$ are complete multipartite graphs with $\beta(H_1)\leq\cdots\leq\beta(H_n)\leq 1/2$, then \[\Phi(H_1\times\cdots\times H_n,\nu)\geq \nu^{\log_{\beta(H_n)}(1-\beta(H_n))}.\] 
\end{corollary}

The final ingredient needed in the proof of Theorem \ref{Thm7} is a result concerning stability of independent sets in direct products of complete multipartite graphs. One of the first instances of such a result is due to the first author, Dinur, Friedgut, and Sudakov \cite{Alon2} and concerns graphs of the form $K_t^n$ (the direct product of $n$ copies of the complete graph $K_t$). They show that the maximum-sized independent sets in such a graph are precisely the sets of vertices obtained by fixing one of the coordinates of the vertices to be constant. Furthermore, they show that every independent set whose size is almost maximal must be close to one of these maximum-sized independent sets. More precisely, they prove the following.  

\begin{theorem}[\!\!\cite{Alon2}]\label{Thm4}
For each integer $t\geq 3$, there exists a constant $M(t)$ with the following property. If $I\subseteq V(K_t^n)$ is an independent set with $\mu(I)=\dfrac{1}{t}(1-\varepsilon)$, then there exists a maximum-sized independent set $J$ such that $\mu(I\Delta J)\leq M(t)\varepsilon$, where $I\Delta J=(I\setminus J)\cup(J\setminus I)$. 
\end{theorem}

Ghandehari and Hatami \cite{Ghandehari} improved upon Theorem \ref{Thm4} and made it explicit by showing that if $t\geq 20$ and $\varepsilon\leq 10^{-9}$, then one can take $M(t)=40/t$. Brakensiek greatly improved upon these results with the following theorem. 

\begin{theorem}[\!\!\!\cite{Brakensiek}]\label{Thm5}
Let $t\geq 3$ be an integer. If $I\subseteq V(K_t^n)$ is an independent set with $\mu(I)=\dfrac{1}{t}(1-\varepsilon)>\dfrac{3t-2}{t^3}$, then there exists a maximum-sized independent set $J$ such that \[\mu(I\setminus J)\leq 4\varepsilon^{\log t/\log(t/(t-1))}.\] 
\end{theorem} 

In order to prove Theorem \ref{Thm7}, we will need to generalize Theorem \ref{Thm5} so that it applies to direct products of balanced complete multipartite graphs that might be of different sizes. First, we fix some notation. If $G=\prod_{i=1}^nH_i$, where the graphs $H_i$ are complete multipartite graphs, we let $X_i(1),\ldots,X_i(t_i)$ be the partite sets of $H_i$. Let \[J_{a,j}=\{(x_1,\ldots,x_n)\in V(G):x_j\in X_j(a)\}.\] Let \[\eta(t)=\frac{\log t}{\log\left(\frac{t}{t-1}\right)}=t\log t+\Theta(\log t)\] and \[\omega(t)=\begin{cases} 37/81-(1/2)(5/81)^{1/\eta(3)}\approx 0.2779, & \mbox{if } t=3; \\ 85/256-(1/3)(7/256)^{1/\eta(4)}\approx 0.1741, & \mbox{if } t=4; \\ \dfrac{4t-3}{t^3} & \mbox{if } t\geq 5 \end{cases}\] for all integers $t\geq 3$. 

\begin{theorem}\label{Thm6}
Let $G=\prod_{i=1}^nK[u_i,t_i]$, where $t_1\geq\cdots\geq t_n\geq 3$. Let $I\subseteq V(G)$ be an independent set with $\mu(I)=\dfrac{1}{t_n}(1-\varepsilon)>\omega(t_n)$. There exist $j\in[n]$ and $a\in[t_j]$ such that \[t_j<\dfrac{t_n}{1-\varepsilon}\quad\text{and}\quad\mu(I\setminus J_{a,j})<4\varepsilon^{\eta(t_n)}.\]  
\end{theorem}

As with Theorem \ref{Thm3}, our proof of Theorem \ref{Thm6} closely follows Brakensiek's arguments from Section 3.2 of \cite{Brakensiek}. We  have attempted to focus on the analysis that is needed to transfer the proofs to the setting in which the graphs in the product are not identical. 

The proofs of Theorem \ref{Thm3}, Corollary \ref{Cor1}, and Theorem \ref{Thm6} are somewhat technical, so we have decided to place them in Sections \ref{Iso} and \ref{Stability}, which are after the proof of Theorem \ref{Thm7}. Finally, we strengthen Burcroff's Conjecture \ref{Conj1} by removing the assumption that the complete multipartite graphs in the direct product are balanced. 

\begin{conjecture}\label{Conj2}
If $H_1,\ldots,H_n$ are complete multipartite graphs and $G=\prod_{i=1}^nH_i$, then \[\alpha(G)=\IR(G).\] 
\end{conjecture}

\section{Near Independence of Irredundant Sets}\label{Near}

There is an alternative characterization of irredundant sets of a graph $G$ that follows immediately from the definition. Specifically, if $S\subseteq V(G)$, then $S$ is irredundant if and only if for every $v\in S$, one of the following holds: 
\begin{enumerate}[(a)]
\item No element of $S$ is adjacent to $v$.
\item There exists $w\in V(G)\setminus S$ such that $v$ is the only neighbor of $w$ in $S$.
\end{enumerate} 
If $S$ is an irredundant set, then we say a vertex $v\in S$ is \emph{lonely} if no element of $S$ is adjacent to $v$. Otherwise, we say $v$ is \emph{social}. If $v$ is social, then it must satisfy condition (b) in the above characterization. In this case, we say the vertex $w$ is a \emph{private neighbor} of $v$. Let $\pn[v;S]$ denote the set of private neighbors of the social vertex $v$. Let $\Lon(S)$ and $\Soc(S)$ denote the set of lonely vertices in $S$ and the set of social vertices in $S$, respectively. Observe that $\Lon(S)$ is an independent set. 

We are now able to prove Theorem \ref{Thm3}, which not only generalizes and improves upon Theorem \ref{Thm1}, but also turns out to be a crucial ingredient in the proof of Theorem \ref{Thm7}. 
\begin{proof}[Proof of Theorem \ref{Thm3}]
Let $H_1,\ldots,H_n$ be complete multipartite graphs, and let $G=\prod_{i=1}^nH_i$. Let $S$ be an irredundant set in $G$. We have seen that $\Lon(S)$ and $\Soc(S)$ form a partition of $S$ and that $\Lon(S)$ is independent. Hence, we need only show that $|\Soc(S)|\leq 2^n$. As in the introduction, let $X_i(1),\ldots,X_i(t_i)$ denote the partite sets of the complete multipartite graph $H_i$. For each vertex $v=(v_1,\ldots,v_n)\in V(G)$, let $c_v(i)$ be the unique integer in $[t_i]$ such that $v_i\in X_i(c_v(i))$. Furthermore, let $f_v(x_1,\ldots,x_n)\in\mathbb Q[x_1,\ldots,x_n]$ be the polynomial defined by $f_v(x_1,\ldots,x_n)=\prod_{i=1}^n(x_i-c_v(i))$. 

For each $v\in\Soc(S)$, choose some vertex $p_v\in\pn[v;S]$. Note that the vertices $p_v$ for $v\in\Soc(S)$ are all distinct by the definition of the sets $\pn[v;S]$. For any $y,z\in\Soc(S)$, we know that $p_y$ is not adjacent to $z$. This means that there is an index $i\in[n]$ such that $c_{p_y}(i)=c_z(i)$, so $f_{p_y}(c_z(1),\ldots,c_z(n))=0$. On the other hand, $f_{p_y}(c_y(1),\ldots,c_y(n))\neq 0$ because $p_y$ is adjacent to $y$. These conditions easily imply that the polynomials $f_{p_y}$ for $y\in\Soc(S)$ are linearly independent. These polynomials are multilinear, so they lie in the $2^n$-dimensional space spanned by the monomials of the form $\prod_{i\in A}x_i$ for $A\subseteq[n]$. This implies that $|\Soc(S)|\leq 2^n$ as desired.       
\end{proof}    

\section{The Proof of Most of Burcroff's Irredundance Conjecture}\label{Proof}

In this section, we prove Theorem \ref{Thm7}. We will need Theorem \ref{Thm3}, which we proved in the previous section, along with Corollary \ref{Cor1} and Theorem \ref{Thm6}, which we prove in the following sections. Recall the definitions of $\eta(t)$ and $\omega(t)$ from the introduction. Note that if $G=\prod_{i=1}^nK[u_i,t_i]$, where $t_1\geq\cdots\geq t_n\geq 2$, then Corollary \ref{Cor1} tells us that 
\begin{equation}\label{Eq32}
\Phi(G,\nu)\geq\nu^{1/\eta(t_n)}
\end{equation}
for all $\nu\in[0,1]$. 

Theorem \ref{Thm7} states that Conjecture \ref{Conj1} holds for all but $37$ exceptional graphs $Z_1,\ldots,Z_{37}$. These exceptional graphs are not necessarily counterexamples to the conjecture; they are simply the graphs that our proof technique cannot handle. These exceptional graphs are the following:

\begin{alignat*}{10}
&&&K_3^4\qquad\quad &&K[2,3]\times K_3^3\qquad\quad &&K[3,3]\times K_3^3\qquad\quad &&K_4\times K_3^3\\              
&&&K[2,4]\times K_3^3\qquad\quad &&K_4\times K[2,3]\times K_3^2\qquad\quad &&K_4^2\times K_3^2\qquad\quad &&K[2,4]\times K_4\times K_3^2\\
&&&K_4^2\times K[2,3]\times K_3\qquad\quad &&K_4^3\times K_3\qquad\quad &&K_5\times K_3^3\qquad\quad &&K[2,5]\times K_3^3\\
&&&K_5\times K[2,3]\times K_3^2\qquad\quad &&K_5\times K_4\times K_3^2\qquad\quad &&K_5\times K_4^2\times K_3\qquad\quad &&K_5^2\times K_3^2 \\
&&&K_6\times K_3^3\qquad\quad &&K_6\times K_4\times K_3^2\qquad\quad
&&K_6\times K_4^2\times K_3\qquad\quad &&K_6\times K_5\times K_3^2\\ 
&&&K_7\times K_3^3\qquad\quad 
&&K_7\times K_4\times K_3^2\qquad\quad &&K_8\times K_3^3\qquad\quad 
&&K_8\times K_4\times K_3^2\\
&&&K_9\times K_3^3\qquad\quad 
&&K_{10}\times K_3^3\qquad\quad 
&&K_3^5\qquad\quad
&&K[2,3]\times K_3^4\\ 
&&&K_4\times K_3^4\qquad\quad 
&&K_4^2\times K_3^3\qquad\quad
&&K_4^3\times K_3^2\qquad\quad 
&&K_5\times K_3^4\\ 
&&&K_5\times K_4\times K_3^3\qquad\quad
&&K_6\times K_3^4\qquad\quad 
&&K_7\times K_3^4\qquad\quad 
&&K_3^6\\
&&&K_4\times K_3^5  && &&  \\
\end{alignat*}

\begin{proof}[Proof of Theorem \ref{Thm7}]
Let $G=\prod_{i=1}^nK[u_i,t_i]$, where $t_1\geq\cdots\geq t_n$. As mentioned in the introduction, this theorem was proven in \cite{defantandiyer} in the cases $t_n\leq 2$ and $n\leq 3$ (although it was Burcroff who observed that the proof showing that $\alpha(G)=\Gamma(G)$ actually proves the stronger fact that $\alpha(G)=\IR(G)$). Hence, we may assume $t_n\geq 3$ and $n\geq 4$. Assume $G$ is not one of the $37$ exceptional graphs listed above. Let $S\subseteq V(G)$ be a maximum-sized irredundant set. We must have $\mu(S)\geq\beta(G)=1/t_n$. 

Consider the set $\Lon(S)$ of lonely vertices in $S$ and the set $\Soc(S)$ of social vertices in $S$, as defined in Section \ref{Near}. Since $\Lon(S)$ is an independent set, we know that $\mu(\Lon(S))\leq 1/t_n$. Write $\mu(\Lon(S))=\dfrac{1}{t_n}(1-\varepsilon)$. Let \[\varepsilon_0=\frac{2^nt_n}{|V(G)|}=\frac{2^n}{u_1\cdots u_nt_1\cdots t_{n-1}}.\] By Theorem \ref{Thm3}, we know that $|\Soc(S)|\leq 2^n$, so 
\begin{equation}\label{Eq34}
\varepsilon=1-t_n\mu(\Lon(S))=1-t_n(\mu(S)-\mu(\Soc(S)))\leq 1-t_n\left(\frac{1}{t_n}-\frac{2^n}{|V(G)|}\right)=\varepsilon_0.
\end{equation} 

Let us assume for the moment that $G\neq K_3^7$; we will return to the case $G=K_3^7$ later. We claim that 
\begin{equation}\label{Eq33}
\frac{1}{t_n}(1-\varepsilon_0)>\omega(t_n).
\end{equation}
We first prove this claim when $t_n\geq 4$. Because $n\geq 4$, we have \[\frac{1}{t_n}(1-\varepsilon_0)=\frac{1}{t_n}-\frac{2^n}{u_1\cdots u_nt_1\cdots t_n}\geq \frac{1}{t_n}-\frac{2^n}{t_n^n}\geq \frac{1}{t_n}-\left(\frac{2}{t_n}\right)^4.\] It is easy to check that $\dfrac{1}{t_n}-\left(\dfrac{2}{t_n}\right)^4>\omega(t_n)$ when $t_n\geq 4$.

We now prove that \eqref{Eq33} holds when $t_n=3$. We wish to see that $\dfrac{1}{3}(1-\varepsilon_0)>\omega(3)$, which we can rewrite as 
\begin{equation}\label{Eq35}
\frac{2^n}{u_1\cdots u_nt_1\cdots t_{n-1}}<1-3\omega(3)\approx 0.166285.
\end{equation} If $n\geq 8$, then \[\frac{2^n}{u_1\cdots u_nt_1\cdots t_{n-1}}\leq\frac{2^n}{3^{n-1}}<1-3\omega(3),\] so we may assume $n\leq 7$. 

Suppose $n=4$. It is easy to check that \eqref{Eq33} holds whenever $t_1\geq 11$ or $u_1\cdots u_n\geq 4$, so we may assume $t_1\leq 10$ and $u_1\cdots u_n\leq 3$. This leaves us with only finitely many graphs. We can now check by hand that among the remaining graphs, the claim fails precisely for those appearing in our list of exceptional graphs. In other words, we can use the assumption that $G$ is not in that list to verify that \eqref{Eq35} holds. 

The proofs of the cases $n=5$, $n=6$, and $n=7$ are similar to the proof of the case $n=4$. For the case $n=7$, we must also use the assumption that $G\neq K_3^7$.  
 
We now know that $\Lon(S)$ is an independent set satisfying \[\mu(\Lon(S))=\frac{1}{t_n}(1-\varepsilon)\geq\frac{1}{t_n}(1-\varepsilon_0)>\omega(t_n),\] so we can apply Theorem \ref{Thm6} to see that there exist $j\in[n]$ and $a\in[t_j]$ such that 
\begin{equation}\label{Eq39}
t_j<\frac{t_n}{1-\varepsilon}\quad\text{and}\quad\mu(\Lon(S)\setminus J_{a,j})<4\varepsilon^{\eta(t_n)}.
\end{equation} 

For every $v\in\Soc(S)$, choose a vertex $p_v\in\pn[v;S]$. Let \[P=\bigcup_{v\in\Soc(S)\cap J_{a,j}}\{p_v\}\quad\text{and}\quad Y=(S\setminus J_{a,j})\cup P.\] If $v\in\Soc(S)$, then $p_v$ is adjacent to $v$. Because $J_{a,j}$ is independent, $P$ is disjoint from $J_{a,j}$. It follows that $Y$ is disjoint from $J_{a,j}$. By the definition of a private neighbor given in Section \ref{Near}, the vertices $p_v$ for $v\in\Soc(S)\cap J_{a,j}$ are distinct and do not lie in $S$. Consequently, \[\mu(Y)=\mu(S\setminus J_{a,j})+\mu(P)=\mu(S\setminus J_{a,j})+\mu(\Soc(S)\cap J_{a,j})\] \[=\mu(\Soc(S)\cup(\Lon(S)\setminus J_{a,j}))=\mu(\Soc(S))+\mu(\Lon(S)\setminus J_{a,j})<\frac{2^n}{|V(G)|}+4\varepsilon^{\eta(t_n)}.\] Using \eqref{Eq34} and the definition of $\varepsilon_0$, we find that 
\begin{equation}\label{Eq36}
\mu(Y)<\varepsilon_0/t_n+4\varepsilon_0^{\eta(t_n)}.
\end{equation}    

Since $Y$ is disjoint from $J_{a,j}$, we have 
\begin{equation}\label{Eq37}
\mu((\partial Y)\cap J_{a,j})\geq\frac{1}{t_j}\mu(\partial Y)\geq\frac{1}{t_j}\mu(Y)^{1/\eta(t_n)},
\end{equation} where we have used Corollary \ref{Cor1} in the form of equation \eqref{Eq32}. By the definition of a private neighbor, $\partial P$ is disjoint from $\Lon(S)\cap J_{a,j}$. We also know that $\partial(\Lon(S)\setminus J_{a,j})$ is disjoint from $\Lon(S)\cap J_{a,j}$, so $(\partial Y)\cap J_{a,j}\subseteq J_{a,j}\setminus \Lon(S)$. Hence, \[\mu((\partial Y)\cap J_{a,j})\leq\mu(J_{a,j})-\mu(\Lon(S)\cap J_{a,j})\leq\frac{1}{t_n}-\mu(\Lon(S)\cap J_{a,j})\leq\mu(S)-\mu(\Lon(S)\cap J_{a,j})\] \begin{equation}\label{Eq42}
=\mu(S\setminus J_{a,j})+\mu(\Soc(S)\cap J_{a,j})=\mu(S\setminus J_{a,j})+\mu(P)=\mu(Y).
\end{equation} 

We wish to show that $S=J_{a,j}$. Assume that this is not the case. Because $\mu(S)\geq 1/t_n\geq\mu(J_{a,j})$, $S$ cannot be a proper subset of $J_{a,j}$. As a consequence, $\mu(Y)\geq\mu(S\setminus J_{a,j})>0$. Therefore, we can combine \eqref{Eq37} and \eqref{Eq42} to see that 
\begin{equation}\label{Eq40}
\frac{1}{t_j}\leq\mu(Y)^{1-1/\eta(t_n)}.
\end{equation}
Note that we have divided each side of an inequality by $\mu(Y)^{1/\eta(t_n)}$; this is precisely where we have used the fact that $\mu(Y)>0$. 
We now use \eqref{Eq34}, \eqref{Eq39}, \eqref{Eq36}, and \eqref{Eq40} to see that \[\frac{1-\varepsilon_0}{t_n}\leq\frac{1-\varepsilon}{t_n}<\frac{1}{t_j}\leq\mu(Y)^{1-1/\eta(t_n)}<\left(\varepsilon_0/t_n+4\varepsilon_0^{\eta(t_n)}\right)^{1-1/\eta(t_n)}
\] \[=\left(\frac{2^n}{u_1\cdots u_nt_1\cdots t_n}+4\left(\frac{2^n}{u_1\cdots u_nt_1\cdots t_{n-1}}\right)^{\eta(t_n)}\right)^{1-1/\eta(t_n)}\leq \left(\frac{2^n}{t_n^n}+4\left(\frac{2^n}{t_n^{n-1}}\right)^{\eta(t_n)}\right)^{1-1/\eta(t_n)}\]
\begin{equation}\label{Eq38}
\leq \left(\frac{16}{t_n^4}+4\left(\frac{16}{t_n^3}\right)^{\eta(t_n)}\right)^{1-1/\eta(t_n)},
\end{equation}
where we have used the fact that $n\geq 4$ in the last step. This tells us that \[1-\varepsilon_0<t_n\left(\frac{16}{t_n^4}+4\left(\frac{16}{t_n^3}\right)^{\eta(t_n)}\right)^{1-1/\eta(t_n)}<(t_n^2)^{1-1/\eta(t_n)}\left(\frac{16}{t_n^4}+4\left(\frac{16}{t_n^3}\right)^{\eta(t_n)}\right)^{1-1/\eta(t_n)}\] \[=\left(\frac{16}{t_n^2}+4t_n^2\left(\frac{16}{t_n^3}\right)^{\eta(t_n)}\right)^{1-1/\eta(t_n)}.\] This last expression is decreasing as a function of $t_n$. If $t_n\geq 5$, then \[1-\varepsilon_0<\left(\frac{16}{5^2}+4\cdot 5^2\left(\frac{16}{5^3}\right)^{\eta(5)}\right)^{1-1/\eta(5)}\approx 0.6809.\] This contradicts the fact that \[\varepsilon_0=\frac{2^n}{u_1\cdots u_nt_1\cdots t_{n-1}}\leq\frac{2^n}{t_n^{n-1}}\leq\frac{2^n}{5^{n-1}}\leq\frac{2^4}{5^3}=0.128.\] Therefore, we may assume $t_n\in\{3,4\}$.

If $t_n=4$, then \eqref{Eq38} tells us that \[\frac{1-\varepsilon_0}{4}<\left(\frac{16}{4^4}+4\left(\frac{16}{4^3}\right)^{\eta(4)}\right)^{1-1/\eta(4)}\approx 0.1181.\] This contradicts the fact that \[\varepsilon_0=\frac{2^n}{u_1\cdots u_nt_1\cdots t_{n-1}}\leq\frac{2^n}{t_n^{n-1}}\leq\frac{2^n}{4^{n-1}}\leq\frac{2^4}{4^3}=0.25.\]
If $t_n=3$, then invoking \eqref{Eq34}, \eqref{Eq33}, and \eqref{Eq39} yields \[t_j<\frac{t_n}{1-\varepsilon}\leq\frac{t_n}{1-\varepsilon_0}<\frac{1}{\omega(3)}<4.\] This tells us that $t_j=3$, so \eqref{Eq40} becomes 
\begin{equation}\label{Eq41}
\frac{1}{3}<\left(\varepsilon_0/3+4\varepsilon_0^{\eta(3)}\right)^{1-1/\eta(3)}.
\end{equation} We saw in \eqref{Eq35} that $\varepsilon_0<0.166285$, which easily contradicts \eqref{Eq41}. 

We have reached our desired contradiction in all cases except that in which $G=K_3^7$. In this case, we have $\mu(\Lon(S))=\frac{1}{3}(1-\varepsilon)\geq\frac{1}{3}(1-\varepsilon_0)=\frac{1}{3}(1-2^7/3^6)>7/27$, so we can apply Theorem \ref{Thm5} to see that \[\mu(\Lon(S)\setminus J_{a,j})<4\varepsilon^{\eta(3)}\] for some $j\in[7]$ and $a\in[3]$. The proof now proceeds exactly as before. We define the set $Y$ as before, assume that $S\neq J_{a,j}$, and deduce that \eqref{Eq40} holds with $t_j=t_n=3$. That is,
 \[\frac{1}{3}<\left(\varepsilon_0/3+4\varepsilon_0^{\eta(3)}\right)^{1-1/\eta(3)}=\left(\left(2^7/3^6\right)/3+4\left(2^7/3^6\right)^{\eta(3)}\right)^{1-1/\eta(3)}\approx 0.2256.\] This is our final contradiction.  
\end{proof}

\begin{remark}
Suppose $G=\prod_{i=1}^nK[u_i,t_i]$ is not one of the $37$ exceptional graphs listed above. The preceding proof of Theorem \ref{Thm7} shows that if $t_1\geq\cdots\geq t_n\geq 3$ and $n\geq 4$, then every irredundant set of $G$ of size $\IR(G)$ is actually an independent set. 
\end{remark}

\section{Vertex Isoperimetry}\label{Iso}
In this section, we prove Theorem \ref{Thm2} and Corollary \ref{Cor1}. Because deducing the corollary from the theorem is quick, we will do this first. 

\begin{proof}[Proof of Corollary \ref{Cor1}]
Assume $H_1,\ldots,H_n$ are complete multipartite graphs such that $\beta(H_1)\leq\cdots\leq\beta(H_n)\leq 1/2$. By Remark \ref{Rem1}, the hypotheses of Theorem \ref{Thm2} are satisfied. The proof of the corollary is by induction on $n$. The case $n=1$ is an immediate consequence of Theorem \ref{Thm2}, so assume $n\geq 2$. The desired inequality is obvious if $\nu=0$, so we can also assume $\nu>0$. 

If $\nu\leq\beta(H_n)$, then it follows from Theorem \ref{Thm2} and induction that \[\Phi(H_1\times\cdots\times H_n,\nu)=(1-\beta(H_n))\Phi\left(H_1\times\cdots\times H_{n-1},\frac{\nu}{\beta(H_n)}\right)\] \[\geq (1-\beta(H_n))\left(\frac{\nu}{\beta(H_n)}\right)^{\log_{\beta(H_{n-1})}(1-\beta(H_{n-1}))}\geq (1-\beta(H_n))\left(\frac{\nu}{\beta(H_n)}\right)^{\log_{\beta(H_n)}(1-\beta(H_n))}\] \[=\nu^{\log_{\beta(H_n)}(1-\beta(H_n))}.\] By a similar token, if $\beta(H_n)<\nu\leq 1$, then \[\Phi(H_1\times\cdots\times H_n,\nu)=1-\beta(H_n)+\beta(H_n)\Phi\left(H_1\times\cdots\times H_{n-1},\frac{\nu-\beta(H_n)}{1-\beta(H_n)}\right)\] \[\geq \beta(H_n)^{\log_{\beta(H_n)}(1-\beta(H_n))}+\beta(H_n)\left(\frac{\nu-\beta(H_n)}{1-\beta(H_n)}\right)^{\log_{\beta(H_{n-1})}(1-\beta(H_{n-1}))}\] \[\geq \beta(H_n)^{\log_{\beta(H_n)}(1-\beta(H_n))}+\beta(H_n)\left(\frac{\nu-\beta(H_n)}{1-\beta(H_n)}\right)^{\log_{\beta(H_n)}(1-\beta(H_n))}.\] To ease notation, put $\beta=\beta(H_n)$, $c=\log_{\beta}(1-\beta)$, and $x=\beta/\nu$. Our assumption on $\nu$ implies that $\beta\leq x<1$. We wish to show that \[\beta^c+\beta\left(\frac{\nu-\beta}{1-\beta}\right)^c\geq\nu^c.\] Dividing each side of this inequality by $\nu^c$, we find that it is equivalent to 
\begin{equation}\label{Eq1}
x^c+\beta\left(\frac{1-x}{1-\beta}\right)^c\geq 1.
\end{equation} Observe that equality holds in \eqref{Eq1} if $x=\beta$ or $x=1$. Noting that $0<c\leq 1$, we find that the left-hand side of \eqref{Eq1} is concave down (or constant if $c=1$) as a function of $x$ in the range $\beta\leq x<1$. Therefore, \eqref{Eq1} holds throughout this range. 
\end{proof}

We now turn our attention to proving Theorem \ref{Thm2}. The theorem is easy if $n=1$, so we can assume $n\geq 2$. Let $H_1,\ldots,H_n$ be as in the statement of the theorem, and let $G=\prod_{i=1}^nH_i$. Let $X_i(1),\ldots,X_i(t_i)$ be the partite sets in $H_i$. We may assume that $|X_i(1)|\geq\cdots\geq |X_i(t_i)|$. Notice that $\beta(H_i)=|X_i(1)|/|V(H_i)|$. 

It will be convenient to work with complete graphs rather than complete multipartite graphs, so we define a map $\coll$ that essentially collapses the partite sets. For each $i$, let $H_i'$ be a copy of the complete graph $K_{t_i}$ with $V(H_i)=[t_i]$. Let $G'=\prod_{i=1}^nH_i'$. Define $\coll_i:V(H_i)\to H_i'$ by declaring that $\coll_i$ sends the elements of $X_i(a)$ to $a$ for every $a\in[t_i]$. Let $\coll:V(G)\to V(G')$ be the product map $\coll=\coll_1\times\cdots\times\coll_n$. We also let $\rho=\coll_{\,*}\mu$ denote the pushforward of the uniform probability measure $\mu$ on $V(G)$ under the map $\coll$. That is, $\rho(T)=\mu(\coll^{-1}(T))$ for all $T\subseteq V(G')$. Alternatively, we can simply define $\rho$ on the singleton sets by \[\rho(\{(a_1,\ldots,a_n)\})=\frac{|X_1(a_1)|\cdots|X_n(a_n)|}{|V(G)|}\] and extend its definition by additivity. 

For every set $T\subseteq V(G)$, we have $\rho(\coll(T))=\mu(\coll^{-1}(\coll(T)))\geq\mu(T)$ and $\rho(\partial\coll(T))=\mu(\coll^{-1}(\partial\coll(T)))=\mu(\partial\coll^{-1}(\coll(T)))$ $=\mu(\partial T)$. It follows that \[\Phi_\mu(G,\nu)=\min\{\rho(\partial S):S\subseteq V(G'),\rho(S)\geq\nu\}.\] In other words, the vertex isoperimetric profile $\Phi_\mu(G,\cdot)$ of $G$ with respect to the uniform measure $\mu$ is the same as the vertex isoperimetric profile $\Phi_\rho(G',\cdot)$ of $G'$ with respect to the measure $\rho$. 

We use the notation $J_{a,i}$ from the introduction for the graph $G'$. More precisely, if $i\in[n]$ and $a\in [t_i]$, we put \[J_{a,i}=[t_1]\times\cdots\times [t_{i-1}]\times\{a\}\times[t_{j+1}]\times\cdots\times[t_n]\subseteq V(G').\] Observe that $\rho(J_{a,i})=|X_i(a)|/|V(H_i)|$; in particular, $\rho(J_{1,i})=\beta(H_i)$. The following proposition is crucial in establishing Theorem \ref{Thm2}.

\begin{proposition}\label{Prop1}
Fix $\nu\in(0,1]$, and choose a set $S\subseteq V(G')$ such that $\rho(S)\geq \nu$ and $\rho(\partial S)=\Phi_\rho(G',\nu)$. Assume that $S$ is chosen to maximize $\rho(S)$. There exists a set $S'\subseteq V(G')$ such that $\rho(S')=\rho(S)$, $\rho(\partial S')=\rho(\partial S)$, and either $S'\subseteq J_{1,n}$ or $J_{1,n}\subseteq S'$.
\end{proposition} 

To prove Proposition \ref{Prop1}, we follow \cite{Brakensiek} and define \emph{compressions}.

\begin{definition}\label{Def1}
For $x=(x_1,\ldots,x_n)\in[t_1]\times\cdots\times[t_n]$, let $x_{\neg i}=(x_1,\ldots,x_{i-1},x_{i+1},\ldots,x_n)$. For $T\subseteq [t_1]\times\cdots\times[t_n]$, define the \emph{compression of $T$ in the} $i^\text{th}$ \emph{coordinate} by \[c_i(T)=\{x\in [t_1]\times\cdots\times [t_n]:x_i\leq|\{y\in T:y_{\neg i}=x_{\neg i}\}|\}.\] The set $T$ is called \emph{compressed} if $c_i(T)=T$ for all $i\in[n]$.
\end{definition}

Brakensiek proves some important facts about compressions that are stated as Remark 2, Claim 5, and Claim 6 in \cite{Brakensiek}. The proofs generalize immediately to our more general setting, so we will not repeat them here. Instead, we state the results in the following lemmas and refer the reader to Brakensiek's paper for the proofs. 

\begin{lemma}\label{Lem1}
If $T\subseteq[t_1]\times\cdots\times[t_n]$, then there is a finite sequence $i_1,\ldots,i_k$ of elements of $[n]$ such that $c_{i_k}\circ c_{i_{k-1}}\circ\cdots\circ c_{i_1}(T)$ is compressed. 
\end{lemma}

\begin{lemma}\label{Lem2}
If $I$ is an independent set in $G'$ and $i\in[n]$, then $c_i(I)$ is also independent. 
\end{lemma}

\begin{lemma}\label{Lem3}
If $T\subseteq V(G')$, then $\rho(\partial c_i(T))\leq\rho(\partial T)$ for all $i\in[n]$. 
\end{lemma}

Invoking Lemmas \ref{Lem1} and \ref{Lem3}, we find that we can assume without loss of generality that the set $S$ in Proposition \ref{Prop1} is compressed. 

Define \[\Pi:V(G')\to\{0,1\}^n\] by requiring that $\Pi(x)_i$ is $0$ if $x_i=1$ and is $1$ otherwise. Because $|X_i(1)|/|V(H_i)|=\beta(H_i)$, we have 
\begin{equation}\label{Eq4}
\rho(\Pi^{-1}(z))=\prod_{z_i=1}\prod_{z_k=0}(1-\beta(H_i))\beta(H_k)
\end{equation} 
for all $z\in\{0,1\}^n$. 

For $z\in\{0,1\}^n$, let $\neg z$ be the Boolean complement of $z$. If $T\subseteq V(G')$ is compressed, then 
\begin{equation}\label{Eq2}
\partial T=\bigcup_{z\in\Pi(T)}\Pi^{-1}(\neg z).
\end{equation} 
This means that 
\begin{equation}\label{Eq3}
\rho(\partial T)=\sum_{z\in\Pi(T)}\rho(\Pi^{-1}(\neg z))=\sum_{z\in\Pi(T)}\prod_{z_i=1}\prod_{z_k=0}\beta(H_i)(1-\beta(H_k)).
\end{equation} 
Consequently, 
\begin{equation}\label{Eq5}
\rho(\partial T)=\rho(\partial\,\Pi^{-1}(\Pi(T))).
\end{equation}

For every $B\subseteq[n]$, define $\sigma_B:\{0,1\}^n\to\{0,1\}^n$ by $\sigma_B(x)_i=x_i$ if $i\not\in B$ and $\sigma_B(x)_i=1-x_i$ if $i\in B$. For $A\subseteq[n-1]$ and $T\subseteq V(G')$, let \[F_A(T)=\{x\in \Pi(S):x_i=0\text{ for all }i\in A, x_n=1,\sigma_{A\cup\{n\}}(x)\not\in\Pi(S)\}.\] Following Brakensiek, we define the \emph{folding} operators $\fold_A$ for all $A\subseteq[n-1]$ by \[\fold_A(T)=\Pi^{-1}((\Pi(T)\setminus F_A(T))\cup\sigma_{A\cup\{n\}}(F_A(T))).\] Note that $\fold_A$ is idempotent in the sense that $\fold_A(\fold_A(T))=\fold_A(T)$. We claim that if $T$ and $\fold_A(T)$ are both compressed, then $\rho(T)\leq\rho(\fold_A(T))$ and $\rho(\partial T)\geq \rho(\partial\fold_A(T))$. First, observe that if $T$ is compressed, then $F_\emptyset(T)=\emptyset$, so $\fold_{\emptyset}(T)=\Pi^{-1}(\Pi(T))$. By \eqref{Eq5}, this proves our claim in the case $A=\emptyset$. Now assume that $A\subseteq[n-1]$ is nonempty. If $z\in F_A(T)$ and we let $z'=\sigma_{A\cup\{n\}}(z)$, then \[\rho(\Pi^{-1}(z))=\prod_{z_i=1}\prod_{z_k=0}(1-\beta(H_i))\beta(H_k)\quad\text{and}\quad\rho(\Pi^{-1}(z'))=\prod_{z'_i=1}\prod_{z'_k=0}(1-\beta(H_i))\beta(H_k)\] by \eqref{Eq4}. Using the hypothesis of Theorem \ref{Thm2}, we deduce that  \[\frac{\rho(\Pi^{-1}(z))}{\rho(\Pi^{-1}(z'))}=\frac{1-\beta(H_n)}{\beta(H_n)}\prod_{k\in A}\frac{\beta(H_k)}{1-\beta(H_k)}\leq 1.\] This shows that $\rho(\Pi^{-1}(z))\leq \rho(\Pi^{-1}(\sigma_{A\cup\{n\}}(z)))$ for $z\in F_A(T)$. Now, \[\rho(T)\leq\rho(\Pi^{-1}(\Pi(T)))=\sum_{z\in\Pi(T)}\rho(\Pi^{-1}(z))=\sum_{z\in\Pi(T)\setminus F_A(T)}\rho(\Pi^{-1}(z))+\sum_{z\in F_A(T)}\rho(\Pi^{-1}(z))\] \[\leq \sum_{z\in\Pi(T)\setminus F_A(T)}\rho(\Pi^{-1}(z))+\sum_{z\in F_A(T)}\rho(\Pi^{-1}(\sigma_{A\cup\{n\}}(z)))=\sum_{z\in(\Pi(T)\setminus F_A(T))\cup \sigma_{A\cup\{n\}}(F_A(T))}\rho(\Pi^{-1}(z))\] \[=\rho(\fold_A(T)).\] By a similar argument, \[\frac{\rho(\Pi^{-1}(\neg z))}{\rho(\Pi^{-1}(\sigma_{A\cup\{n\}}(\neg z)))}=\frac{\beta(H_n)}{1-\beta(H_n)}\prod_{k\in A}\frac{1-\beta(H_k)}{\beta(H_k)}\geq 1\] when $z\in F_A(T)$. By \eqref{Eq3}, \[\rho(\partial T)=\sum_{z\in\Pi(T)}\rho(\Pi^{-1}(\neg z))=\sum_{z\in\Pi(T)\setminus F_A(T)}\rho(\Pi^{-1}(\neg z))+\sum_{z\in F_A(T)}\rho(\Pi^{-1}(\neg z))\] \[\geq \sum_{z\in\Pi(T)\setminus F_A(T)}\rho(\Pi^{-1}(\neg z))+\sum_{z\in F_A(T)}\rho(\Pi^{-1}(\sigma_{A\cup\{n\}}(\neg z)))=\sum_{z\in(\Pi(T)\setminus F_A(T))\cup \sigma_{A\cup\{n\}}(F_A(T))}\rho(\Pi^{-1}(\neg z))\] \[=\rho(\partial \fold_A(T)).\]
This completes the proof of our claim, so we can return to our set $S$ and the proof of Proposition \ref{Prop1}. 

\begin{proof}[Proof of Proposition \ref{Prop1}] 
Using \eqref{Eq5} and our assumption that $S$ was chosen to maximize $\rho(S)$, we see that $S=\Pi^{-1}(\Pi(S))$. We claim that there is a sequence $A_1,\ldots,A_\ell$ of subsets of $[n-1]$ and a sequence $S=S_0,S_1,\ldots,S_\ell$ of compressed subsets of $V(G')$ such that $S_i=\fold_{A_i}(S_{i-1})$ for all $i\in[\ell]$ and $\fold_A(S_\ell)=S_\ell$ for all $A\subseteq[n-1]$.
We omit the proof of this claim because it is identical to the proof of Claim 18 and the discussion thereafter in \cite{Brakensiek}. Let $S'=S_\ell$. By the preceding discussion, we know that $\rho(S_\ell)\geq\rho(S_{\ell-1})\geq\cdots\geq\rho(S_0)=\rho(S)$ and $\rho(\partial S_{\ell})\leq\rho(\partial S_{\ell-1})\leq\cdots\leq\rho(\partial S_0)=\rho(\partial S)$. By our choice of $S$, this means that $\rho(S')=\rho(S)\geq \nu$ and $\rho(\partial S')=\rho(\partial S)=\Phi_\rho(G',\nu)$. We want to prove that either $S'\subseteq J_{1,n}$ or $J_{1,n}\subseteq S'$. Suppose $S'\not\subseteq J_{1,n}$ so that there exists $x\in S'\setminus J_{1,n}$. Let $A=\{i\in[n]:x_i=1\}=\{i\in[n]:\Pi(x)_i=0\}$. Because $x\not\in J_{1,n}$, we know that $n\not\in A$. The fact that $\fold_A(S')=S'$ tells us that $F_A(S')=\emptyset$. In particular, $x\not\in F_A(S')$. By the definition of $F_A(S')$, this means that the vector $y=\sigma_{A\cup\{n\}}(\Pi(x))$ is in $\Pi(S')$. However, $y_i=1$ for all $i\in [n-1]$ while $y_n=0$. It now follows easily from the fact that $S'$ is compressed that $J_{1,n}\subseteq S'$. 
\end{proof}

\begin{proof}[Proof of Theorem \ref{Thm2}]
As mentioned above, we may assume $n\geq 2$. Fix $\nu\in(0,1]$, and choose $S\subseteq V(G')$ such that $\rho(S)\geq\nu$ and $\rho(\partial S)=\Phi_\mu(G,\nu)=\Phi_\rho(G',\nu)$. We may assume that $S$ is chosen to maximize $\rho(S)$. By Proposition \ref{Prop1}, we may further assume that either $S\subseteq J_{1,n}$ or $J_{1,n}\subseteq S$. By abuse of notation, we let $\rho$ denote the pushforward of $\mu$ under the collapsing map from $V(H_1\times\cdots\times H_{n-1})$ to $V(H_1'\times\cdots\times H_{n-1}')$ (just as we defined $\rho$ on $V(G')$). 

Assume first that $\nu\leq\beta(H_n)$. We know that $\rho(J_{1,n})=\beta(H_n)$, so $\rho(\partial J_{1,n})\geq\Phi_\rho(G',\nu)=\Phi_\mu(G,\nu)$. It is easy to check that the proper containment $J_{1,n}\subsetneq S$ would imply the contradiction $\rho(\partial S)>\rho(\partial J_{1,n})$. Therefore, $S\subseteq J_{1,n}$. Let \[T=\{(x_1,\ldots,x_{n-1})\in V(H_1'\times\cdots\times H_{n-1}'):(x_1,\ldots,x_{n-1},1)\in S\}.\]  We have $\rho(T)=\rho(S)/\beta(H_n)\geq\nu/\beta(H_n)$, so \[\Phi_\mu(G,\nu)=\Phi_\rho(G',\nu)=\rho(\partial S)=(1-\beta(H_n))\rho(\partial T)\geq(1-\beta(H_n))\Phi_\rho\left(H_1'\times\cdots\times H_{n-1}',\frac{\nu}{\beta(H_n)}\right)\] \[=(1-\beta(H_n))\Phi_\mu\left(H_1\times\cdots\times H_{n-1},\frac{\nu}{\beta(H_n)}\right).\] On the other hand, there exists $T'\subseteq V(H_1'\times\cdots\times H_{n-1}')$ with $\rho(T')\geq \nu/\beta(H_n)$ and \[\rho(\partial T')=\Phi_\rho\left(H_1'\times\cdots\times H_{n-1}',\frac{\nu}{\beta(H_n)}\right)=\Phi_\mu\left(H_1\times\cdots\times H_{n-1},\frac{\nu}{\beta(H_n)}\right).\] Defining \[S'=\{(x_1,\ldots,x_{n-1},1):(x_1,\ldots,x_{n-1})\in T'\},\] we find that $\rho(S')=\beta(H_n)\rho(T')\geq\nu$ and \[\Phi_\mu(G,\nu)=\Phi_\rho(G',\nu)\leq\rho(\partial S')=(1-\beta(H_n))\rho(\partial T')=(1-\beta(H_n))\Phi_\mu\left(H_1\times\cdots\times H_{n-1},\frac{\nu}{\beta(H_n)}\right).\] This completes the proof in the case where $\nu\leq\beta(H_n)$. 

Assume now that $\beta(H_n)<\nu\leq 1$. We must have $J_{1,n}\subseteq S$. Let \[U=\{(x_1,\ldots,x_{n-1})\in V(H_1'\times\cdots\times H_{n-1}'):(x_1,\ldots,x_{n-1},y)\in S\text{ for some }y\in\{2,\ldots,t_n\}\}.\] If $(x_1,\ldots,x_{n-1})\in U$, then $(x_1,\ldots,x_{n-1},z)\in S$ for \emph{all} $z\in\{2,\ldots,t_n\}$. Indeed, adding the additional points of the form $(x_1,\ldots,x_{n-1},z)$ to $S$ increases $\rho(S)$ while keeping $\rho(\partial S)$ the same, so the claim follows from our assumption that $S$ was chosen to maximize $\rho(S)$. We have \[\rho(U)=\frac{\rho(S)-\rho(J_{1,n})}{1-\beta(H_n)}\geq\frac{\nu-\beta(H_n)}{1-\beta(H_n)}\] and \[\Phi_\mu(G,\nu)=\Phi_\rho(G',\nu)=\rho(\partial S)=\rho(\partial J_{1,n})+\rho(\partial U\times\{1\})=1-\beta(H_n)+\beta(H_n)\rho(\partial U)\] \[\geq 1-\beta(H_n)+\beta(H_n)\Phi_\rho\left(H_1'\times\cdots\times H_{n-1}',\frac{\nu-\beta(H_n)}{1-\beta(H_n)}\right)\] \[=1-\beta(H_n)+\beta(H_n)\Phi_\mu\left(H_1\times\cdots\times H_{n-1},\frac{\nu-\beta(H_n)}{1-\beta(H_n)}\right).\]
On the other hand, there exists $U'\subseteq V(H_1'\times\cdots\times H_{n-1}')$ with $\rho(U')\geq \dfrac{\nu-\beta(H_n)}{1-\beta(H_n)}$ and \[\rho(\partial U')=\Phi_\rho\left(H_1'\times\cdots\times H_{n-1}',\frac{\nu-\beta(H_n)}{1-\beta(H_n)}\right)=\Phi_\mu\left(H_1\times\cdots\times H_{n-1},\frac{\nu-\beta(H_n)}{1-\beta(H_n)}\right).\] Defining \[Q=\{(x_1,\ldots,x_{n-1},z):(x_1,\ldots,x_{n-1})\in U', z\in\{2,\ldots,t_n\}\}\quad \text{and} \quad S'=J_{1,n}\cup Q,\]
we find that $\rho(S')=\beta(H_n)+(1-\beta(H_n))\rho(U')\geq \nu$ and  \[\Phi_\mu(G,\nu)=\Phi_\rho(G',\nu)\leq \rho(\partial S')=\rho(\partial J_{1,n})+\rho(\partial Q\cap J_{1,n})=1-\beta(H_n)+\rho(\partial U'\times\{1\})\] \[=1-\beta(H_n)+\beta(H_n)\rho(\partial U')=1-\beta(H_n)+\beta(H_n)\Phi_\mu\left(H_1\times\cdots\times H_{n-1},\frac{\nu-\beta(H_n)}{1-\beta(H_n)}\right).\] This proves the case in which $\beta(H_n)<\nu\leq 1$. 
\end{proof}

\section{Independent Set Stability}\label{Stability}
This section is devoted to proving Theorem \ref{Thm6}. Recall the definitions of $\eta(t)$, $\omega(t)$, and $J_{a,j}$ from the introduction and Definition \ref{Def1} from the previous section. Suppose $H_1,\ldots,H_n$ are complete multipartite graphs such that each graph $H_i$ has $t_i$ partite sets, and let $G=\prod_{i=1}^nH_i$. Say a set $S\subseteq V(G)$ is \emph{sorted} if $\mu(S\cap J_{1,j})\geq\cdots\geq\mu(S\cap J_{t_j,j})$ for all $j\in[n]$. We will often assume the independent sets we consider are sorted. This is simply for notational convenience; we can always relabel the partite sets without loss of generality in order to ensure that the set under consideration is sorted. Most of the results in this section concern large independent sets in direct product graphs. However, we start with a result about independent sets in direct products of complete graphs that makes no assumption on the size of the independent set. 

\begin{proposition}\label{Prop2}
Let $G=\prod_{i=1}^nK_{t_i}$, where $t_1\geq\cdots\geq t_n\geq 3$. Let $I\subseteq V(G)$ be a sorted independent set with $\mu(I)=\dfrac{1}{t_n}(1-\varepsilon)$. Choose $j\in[n]$, and let $\delta=\mu(I\setminus J_{1,j})$. We have \[\varepsilon\geq 1-\frac{t_n}{t_j}-\delta t_n+\frac{t_n}{t_j-1}\left(\frac{\delta}{t_j-1}\right)^{1/\eta(t_n)}.\]
\end{proposition}
\begin{proof}
First, note that \[\mu(I\cap J_{2,j})\geq\frac{\delta}{t_j-1}.\] Using Corollary \ref{Cor1} (in the form of \eqref{Eq32}), we find that \[\mu(\partial(I\cap J_{2,j})\cap J_{1,j})=\frac{1}{t_j-1}\mu(\partial(I\cap J_{2,j}))\geq\frac{1}{t_j-1}\left(\frac{\delta}{t_j-1}\right)^{1/\eta(t_n)}.\] Because $I$ is independent, $I\cap J_{1,j}$ is disjoint from $\partial(I\cap J_{2,j})\cap J_{1,j}$. Thus, \[\frac{1}{t_n}(1-\varepsilon)-\delta=\mu(I)-\mu(I\setminus J_{1,j})=\mu(I\cap J_{1,j})\leq\mu(J_{1,j})-\mu(\partial(I\cap J_{2,j})\cap J_{1,j})\] \[\leq\frac{1}{t_j}-\frac{1}{t_j-1}\left(\frac{\delta}{t_j-1}\right)^{1/\eta(t_n)}.\] Rearranging the inequality $\dfrac{1}{t_n}(1-\varepsilon)-\delta\leq\dfrac{1}{t_j}-\dfrac{1}{t_j-1}\left(\dfrac{\delta}{t_j-1}\right)^{1/\eta(t_n)}$ yields the desired result.   
\end{proof}

In the following lemmas, we assume the independent set from Proposition \ref{Prop2} is large. In Lemma \ref{Lem4}, we find that for every choice of $j$, the value of $\delta$ must either be somewhat large or somewhat small. Lemma \ref{Lem5} shows that if the independent set is compressed, then it cannot be the case that $\delta$ is somewhat large for every choice of $j$. Consequently, in this case, there is some choice of $j$ that makes $\delta$ somewhat small. Lemma \ref{Lem7} is a purely technical result that allows us to prove Lemma \ref{Lem6}, where we remove the hypothesis from Lemma \ref{Lem5} that the independent set is compressed. Finally, we use Proposition \ref{Prop2} to show that if $\delta$ is somewhat small, then it is actually very small. This allows us to complete the proof of Theorem \ref{Thm6}. Many of the ideas below are adapted from Brakensiek's arguments in Section 3.2 of \cite{Brakensiek}. 

\begin{lemma}\label{Lem4}
Let $G=\prod_{i=1}^nK_{t_i}$, where $t_1\geq\cdots\geq t_n\geq 3$. Let $I\subseteq V(G)$ be a sorted independent set such that $\mu(I)>\omega(t_n)$. For all $j\in[n]$, either \[\mu(I\setminus J_{1,j})<\frac{t_j-1}{t_j^5}\quad\text{or}\quad\mu(I\setminus J_{1,j})>\frac{(2t_j-1)(t_j-1)}{t_j^4}.\]
\end{lemma}

\begin{proof}
Let $\delta=\mu(I\setminus J_{1,j})$. The first part of the proof essentially follows Brakensiek's proof of Claim 13 in \cite{Brakensiek} with minor modifications, so we omit the details. Following his argument (and using Corollary \ref{Cor1}), we arrive at the inequality 
\begin{equation}\label{Eq6}
\frac{1}{t_j}+\delta-\frac{1}{t_j-1}\left(\frac{\delta}{t_j-1}\right)^{1/\eta(t_n)}>\omega(t_n).
\end{equation}
We wish to show that this inequality fails for $\dfrac{t_j-1}{t_j^5}\leq\delta\leq\dfrac{(2t_j-1)(t_j-1)}{t_j^4}$. Because the left-hand side of \eqref{Eq6} is concave up as a function of $\delta$ when $\delta$ is in this range, it suffices to prove that the inequality fails when $\delta=\dfrac{t_j-1}{t_j^5}$ and when $\delta=\dfrac{(2t_j-1)(t_j-1)}{t_j^4}$. Replacing $t_j$ by the continuous variable $x$ and recalling that $t_j\geq t_n$, we see that is suffices to prove that  
\begin{equation}\label{Eq7}
\frac{1}{x}+\frac{x-1}{x^5}-\frac{1}{x-1}\left(\frac{1}{x^5}\right)^{1/\eta(t_n)}\leq\omega(t_n)
\end{equation} and 
\begin{equation}\label{Eq8}
\frac{1}{x}+\frac{(2x-1)(x-1)}{x^4}-\frac{1}{x-1}\left(\frac{2x-1}{x^4}\right)^{1/\eta(t_n)}\leq\omega(t_n)
\end{equation} whenever $x\geq t_n$. This is straightforward when $t_n=3$ or $t_n=4$, so we may assume $t_n\geq 5$. Let us differentiate the left-hand sides of \eqref{Eq7} and \eqref{Eq8} with respect to $x$. We check that these derivatives are negative so that the left-hand sides of these inequalities are decreasing in $x$. This means that it suffices to prove them in the case $x=t_n$. Under this assumption, \eqref{Eq7} and \eqref{Eq8} become \[\frac{1}{t_n}+\frac{t_n-1}{t_n^5}-\frac{(t_n-1)^4}{t_n^5}\leq\frac{4t_n-3}{t_n^3}\] and \[\frac{1}{t_n}+\frac{(2t_n-1)(t_n-1)}{t_n^4}-\frac{(t_n-1)^2}{t_n^3}\left(2-\frac{1}{t_n}\right)^{1/\eta(t_n)}\leq\frac{4t_n-3}{t_n^3}.\] Both of these inequalities are easy to verify (for the second, note that $(2-1/t_n)^{1/\eta(t_n)}>1$). 
\end{proof}

\begin{lemma}\label{Lem5}
Let $G=\prod_{i=1}^nK_{t_i}$, where $t_1\geq\cdots\geq t_n\geq 3$. Let $I\subseteq V(G)$ be a compressed independent set such that $\mu(I)>\omega(t_n)$. There exists $j\in[n]$ such that \[\mu(I\setminus J_{1,j})<\frac{t_j-1}{t_j^5}.\]
\end{lemma}

\begin{proof}
The beginning of the proof follows Brakensiek's proof of Lemma 15 in \cite{Brakensiek}. We induct on $n$. If $n=1$, then we are done because $I=\{1\}=J_{1,1}$. Assume that $n\geq 2$ and that the lemma holds for all smaller values of $n$. Let $G'=\prod_{i=1}^{n-1}K_{t_i}$. Let $J_{a,i}'=[t_1]\times\cdots\times[t_{i-1}]\times\{a\}\times[t_{i+1}]\times\cdots\times[t_{n-1}]$. By way of contradiction, assume that $\mu(I\setminus J_{1,j})\geq\dfrac{t_j-1}{t_j^5}$ for all $j\in [n]$. According to Lemma \ref{Lem4}, this implies that 
\begin{equation}\label{Eq10}
\mu(I\setminus J_{1,j})\geq\frac{(2t_j-1)(t_j-1)}{t_j^4}
\end{equation} for all $j\in[n]$. 

For $b\in[t_n]$, put \[I_b=\{(x_1,\ldots,x_{n-1})\in V(G'):(x_1,\ldots,x_{n-1},b)\in I\}.\] The sets $I_b$ are compressed because $I$ is compressed. Choose $a\in[t_{n-1}]$ such that $\mu((I_1\setminus I_2)\cap J_{a,n-1})$ is maximal. Let $\widehat I=I_2\cup((I_1\setminus I_2)\cap J_{a,n-1})$. If we follow Brakensiek's argument \emph{mutatis mutandis}, we find that $\widehat I$ is an independent set. Moreover, $I_2\subseteq\widehat I\subseteq I_1$. By Lemmas \ref{Lem1} and \ref{Lem2}, we can repeatedly apply compressions to the set $\widehat I$ until we obtain a compressed independent set $\widetilde I\subseteq V(G')$. Because $I_1$ and $I_2$ are already compressed, we know that $I_2\subseteq \widetilde I\subseteq I_1$. Now,  
\begin{equation}\label{Eq9}
\mu\left(\widetilde I\right)=\mu\left(\widehat I\right)\geq\mu(I_2)+\frac{\mu(I_1)-\mu(I_2)}{t_{n-1}}=\frac{\mu(I_1)+(t_{n-1}-1)\mu(I_2)}{t_{n-1}}\geq\frac{\mu(I_1)+(t_n-1)\mu(I_2)}{t_{n-1}}.
\end{equation} Because $I$ is compressed, we have $\mu(I_2)\geq\mu(I_3)\geq\cdots\geq\mu(I_{t_n})$. Thus, $\mu(I_1)+(t_n-1)\mu(I_2)\geq\sum_{a=1}^{t_n}\mu(I_a)=t_n\mu(I)$. Combining this with \eqref{Eq9} yields \[\mu\left(\widetilde I\right)\geq\frac{t_n\mu(I)}{t_{n-1}}\geq\mu(I)>\omega(t_n)\geq\omega(t_{n-1}).\] We can now apply our induction hypothesis to the compressed independent set $\widetilde I$ to see that there exists $j\in[n-1]$ such that \[\mu(\widetilde I\setminus J_{1,j}')<\frac{t_j-1}{t_j^5}.\] Because $I$ is compressed and $I_2\subseteq \widetilde I$, we have \[\mu(I\setminus (J_{1,j}\cup J_{1,n}))=\frac{1}{t_n}\sum_{b=2}^{t_n}\mu(I_b\setminus J_{1,j}')\leq \frac{t_n-1}{t_n}\mu(I_2\setminus J_{1,j}')\leq \frac{t_n-1}{t_n}\mu(\widetilde I\setminus J_{1,j}')<\frac{t_n-1}{t_n}\cdot\frac{t_j-1}{t_j^5}.\]
Invoking \eqref{Eq10}, we obtain the inequalities 
\begin{equation}\label{Eq12}
\mu((I\setminus J_{1,j})\cap J_{1,n})=\mu(I\setminus J_{1,j})-\mu(I\setminus(J_{1,j}\cup J_{1,n}))\geq \frac{(2t_j-1)(t_j-1)}{t_j^4}-\frac{t_n-1}{t_n}\cdot\frac{t_j-1}{t_j^5}.
\end{equation} 
and
\begin{equation}\label{Eq11}
\mu((I\setminus J_{1,n})\cap J_{1,j})=\mu(I\setminus J_{1,n})-\mu(I\setminus(J_{1,j}\cup J_{1,n}))\geq \frac{(2t_n-1)(t_n-1)}{t_n^4}-\frac{t_n-1}{t_n}\cdot\frac{t_j-1}{t_j^5}
\end{equation}
Put $I'=I\cap J_{2,j}\cap J_{1,n}$ and $I''=I\cap J_{1,j}\cap J_{2,n}$. Because $I$ is compressed, 
\begin{equation}\label{Eq15}
\mu(I')\geq\frac{1}{t_j-1}\mu((I\setminus J_{1,j})\cap J_{1,n})\geq \frac{2t_j-1}{t_j^4}-\frac{t_n-1}{t_nt_j^5}>\frac{1}{t_j^3}
\end{equation}
and
\begin{equation}\label{Eq16}
\mu(I'')\geq\frac{1}{t_n-1}\mu((I\setminus J_{1,n})\cap J_{1,j})\geq \frac{2t_n-1}{t_n^4}-\frac{t_j-1}{t_nt_j^5}.
\end{equation}
The elements of $I'$ have constant $j^\text{th}$ and $n^\text{th}$ coordinates, so 
\begin{equation}\label{Eq13}
\mu(\partial I'\cap J_{1,j}\cap J_{2,n})=\frac{1}{(t_j-1)(t_n-1)}\mu(\partial I')\geq\frac{1}{(t_j-1)(t_n-1)}\Phi(G,\mu(I')).
\end{equation}
Finally, observe that 
\begin{equation}\label{Eq14}
\mu(I'')+\mu(\partial I'\cap J_{1,j}\cap J_{2,n})\leq\mu(J_{1,j}\cap J_{2,n})=\frac{1}{t_jt_n}
\end{equation} because $I'\cup I''$ is an independent set. We now combine \eqref{Eq15}, \eqref{Eq16}, \eqref{Eq13}, \eqref{Eq14}, and Corollary \ref{Cor1} to obtain
\[\frac{1}{t_jt_n}\geq\mu(I'')+\mu(\partial I'\cap J_{1,j}\cap J_{2,n})\geq \frac{2t_n-1}{t_n^4}-\frac{t_j-1}{t_nt_j^5}+\frac{1}{(t_j-1)(t_n-1)}\Phi(G,\mu(I'))\] \[\geq\frac{2t_n-1}{t_n^4}-\frac{t_j-1}{t_nt_j^5}+\frac{1}{(t_j-1)(t_n-1)}\mu(I')^{1/\eta(t_n)}\] \[\geq\frac{2t_n-1}{t_n^4}-\frac{1}{t_nt_j^4}+\frac{1}{(t_j-1)(t_n-1)}\left(\frac{1}{t_j^3}\right)^{1/\eta(t_n)}.\] 

We seek a contradiction, so our goal is to prove that \[\frac{2t_n-1}{t_n^4}-\frac{1}{t_nt_j^4}+\frac{1}{(t_j-1)(t_n-1)}\left(\frac{1}{t_j^3}\right)^{1/\eta(t_n)}>\frac{1}{t_jt_n}.\] Multiplying both sides of this inequality by $t_jt_n$ yields 
\begin{equation}\label{Eq17}
t_j\frac{2t_n-1}{t_n^3}-\frac{1}{t_j^3}+\frac{t_jt_n}{(t_j-1)(t_n-1)}\left(\frac{1}{t_j^3}\right)^{1/\eta(t_n)}>1.
\end{equation}
It is straightforward (though somewhat tedious) to verify that \eqref{Eq17} holds for each fixed $t_n\in\{3,\ldots,21\}$, so we may assume $t_n\geq 22$. To ease notation, let $Q(t_j,t_n)$ denote the left-hand side of \eqref{Eq17}. If we fix $t_n$ and replace $t_j$ with a continuous variable $x\geq t_n$, then we can differentiate $Q(x,t_n)$ with respect to $x$ and find (after some simplifying) that \[\frac{\partial}{\partial x}Q(x,t_n)=\frac{2t_n-1}{t_n^3}+\frac{3}{x^4}-\frac{t_n}{t_n-1}\frac{x^{-3/\eta(t_n)}}{(x-1)^2}(1+(3/\eta(t_n))(x-1))\] \[>\frac{2t_n-1}{t_n^3}-\frac{t_n}{t_n-1}\frac{x^{-3/\eta(t_n)}}{(x-1)^2}(1+(3/\eta(t_n))(x-1)).\] This last expression is increasing as a function of $x$, so we obtain a lower bound for $\dfrac{\partial}{\partial x}Q(x,t_n)$ by evaluating that expression when $x=t_n$. More precisely, \[\frac{\partial}{\partial x}Q(x,t_n)>\frac{2t_n-1}{t_n^3}-\frac{t_n}{t_n-1}\frac{t_n^{-3/\eta(t_n)}}{(t_n-1)^2}(1+(3/\eta(t_n))(t_n-1))=\frac{2t_n-1}{t_n^3}-\frac{1}{t_n^2}(1+(3/\eta(t_n))(t_n-1))\] \[=\frac{t_n-1}{t_n^2}\left(\frac{1}{t_n}-3/\eta(t_n)\right)>0,\] where the last inequality uses the assumption that $t_n\geq 22$ and is easy to verify. We now know that the left-hand side of \eqref{Eq17} is increasing as a function of $t_j$ when $t_n$ is fixed, so we are left to prove \eqref{Eq17} when $t_j=t_n$. With this substitution, \eqref{Eq17} becomes \[\frac{2t_n-1}{t_n^2}-\frac{1}{t_n^3}+\frac{t_n-1}{t_n}>1,\] which is certainly true. 
\end{proof}

The next lemma is purely technical and serves no purpose for us other than allowing us to prove Lemma \ref{Lem6}. 
\begin{lemma}\label{Lem7}
If $t\geq 3$ is an integer and $x,\nu,m$ are real numbers such that \[m/2\geq\nu\geq\frac{2x-1}{x^4},\quad x\geq t, \quad\text{and}\quad m\geq\omega(t),\] then 
\begin{equation}\label{Eq24}
\frac{1}{x}+\frac{1}{x^4}+\nu-\frac{1}{x-1}\nu^{1/\eta(t)}\leq m.
\end{equation}
\end{lemma}
\begin{proof}
Let us first assume $\nu\geq \dfrac{2t-1}{t^3}$. Since $\nu\leq m/2$, it suffices to prove that \[\frac{1}{x}+\frac{1}{x^4}+\nu-\frac{1}{x-1}\nu^{1/\eta(t)}\leq 2\nu,\] which is equivalent to \[\frac{1}{x}+\frac{1}{x^4}-\nu-\frac{1}{x-1}\nu^{1/\eta(t)}\leq 0.\] We will prove the stronger inequality \[\frac{1}{x}+\frac{1}{x^4}-\nu-\frac{1}{x}\nu^{1/\eta(t)}\leq 0.\] Because \[-\nu-\frac{1}{x}\nu^{1/\eta(t)}\leq -\frac{2t-1}{t^3}-\frac{1}{x}\left(\frac{2t-1}{t^3}\right)^{1/\eta(t)},\] we wish to show that \[\frac{1}{x}\left(1-\left(\frac{2t-1}{t^3}\right)^{1/\eta(t)}\right)+\frac{1}{x^4}-\frac{2t-1}{t^3}\leq 0.\] The left-hand side of this last inequality is decreasing as a function of $x$, so it suffices to prove that it holds when $x=t$. In this case, the inequality becomes \[\frac{1}{t}\left(1-\left(\frac{2t-1}{t^3}\right)^{1/\eta(t)}\right)+\frac{1}{t^4}-\frac{2t-1}{t^3}\leq 0,\] which one can verify is true for all $t\geq 3$. 

We are now left to prove that \eqref{Eq24} holds when $\dfrac{2x-1}{x^4}\leq\nu\leq \dfrac{2t-1}{t^3}$. We will prove the stronger inequality 
\begin{equation}\label{Eq25}
\frac{1}{x}+\frac{1}{x^4}+\nu-\frac{1}{x-1}\nu^{1/\eta(t)}\leq \omega(t).
\end{equation}
When viewed as a function of $\nu$, the left-hand side of \eqref{Eq25} is concave up. Hence, it suffices to prove \eqref{Eq25} when $\nu=\dfrac{2x-1}{x^4}$ and when $\nu=\dfrac{2t-1}{t^3}$. 

First, assume $\nu=\dfrac{2t-1}{t^3}$. It is straightforward to verify \eqref{Eq25} for $t=3$ and $t=4$, so assume $t\geq 5$. We will prove the stronger inequality 
\begin{equation}\label{Eq26}
\frac{1}{x}+\frac{1}{x^4}+\nu-\frac{1}{x}\nu^{1/\eta(t)}\leq \omega(t).
\end{equation} The left-hand side of this last inequality is decreasing as a function of $x$, so it suffices to prove it when $x=t$. With the substitutions $x=t$ and $\nu=\dfrac{2t-1}{t^3}$, \eqref{Eq26} becomes \[\frac{1}{t}+\frac{1}{t^4}+\frac{2t-1}{t^3}-\frac{1}{t}\left(\frac{2t-1}{t^3}\right)^{1/\eta(t)}\leq \frac{4t-3}{t^3},\] where we have used the fact that $\omega(t)=\dfrac{4t-3}{t^3}$ for $t\geq 5$. One can verify that this last inequality holds for all $t\geq 5$. 

Finally, we must prove that \eqref{Eq25} holds when $\nu=\dfrac{2x-1}{x^4}$. With this substitution, \eqref{Eq25} becomes \[\frac{1}{x}+\frac{1}{x^4}+\frac{2x-1}{x^4}-\frac{1}{x-1}\left(\frac{2x-1}{x^4}\right)^{1/\eta(t)}\leq\omega(t).\] One can verify this last inequality when $t=3$ and $t=4$, so we may assume $t\geq 5$. We will prove the stronger inequality 
\begin{equation}\label{Eq27}
\frac{1}{x}+\frac{2}{x^3}-\frac{1}{x}\left(\frac{1}{x^3}\right)^{1/\eta(t)}\leq\frac{4t-3}{t^3}.
\end{equation} When $x=t$, \eqref{Eq27} becomes \[\frac{1}{t}+\frac{2}{t^3}-\frac{(t-1)^3}{t^4}\leq\frac{4t-3}{t^3},\] and it is easy to check that this holds for all $t\geq 5$. Thus, it suffices to prove that the left-hand side of \eqref{Eq27} is decreasing as a function of $x$. The derivative of the left-hand side of \eqref{Eq27} with respect to $x$ is \[-x^{-2}-6x^{-4}+\left(3/\eta(t)+1\right)x^{-3/\eta(t)-2},\] which is less than \[x^{-2}\left(-1+\left(3/\eta(t)+1\right)x^{-3/\eta(t)}\right).\] In order to prove that this derivative is negative, we need only show that $(3/\eta(t)+1)x^{-3/\eta(t)}<1$. This follows from the fact that \[\left(1+\frac{3}{\eta(t)}\right)^{\eta(t)}<e^3<x^3. \qedhere\]  
\end{proof}

\begin{lemma}\label{Lem6}
Let $G=\prod_{i=1}^nK_{t_i}$, where $t_1\geq\cdots\geq t_n\geq 3$. Let $I\subseteq V(G)$ be a sorted independent set such that $\mu(I)>\omega(t_n)$. There exists $j\in[n]$ such that \[\mu(I\setminus J_{1,j})<\frac{t_j-1}{t_j^5}.\]
\end{lemma}
\begin{proof}
Our proof follows Brakensiek's proof of Lemma 16 in \cite{Brakensiek}. Assume that the lemma is false, and deduce from Lemma \ref{Lem4} that \[\mu(I\setminus J_{1,j})>\frac{(2t_j-1)(t_j-1)}{t_j^4}\] for all $j\in[n]$. Whenever $i\neq j$, we have \[\mu(c_i(I)\setminus J_{1,j})=\mu(I\setminus J_{1,j})>\frac{(2t_j-1)(t_j-1)}{t_j^4}.\] If 
\begin{equation}\label{Eq19}
\mu(c_j(I)\setminus J_{1,j})>\frac{(2t_j-1)(t_j-1)}{t_j^4}
\end{equation}
for all $j\in [n]$, then we can use Lemmas \ref{Lem1} and \ref{Lem2} to obtain a compressed independent set $I'$ with $\mu(I')=\mu(I)$ and $\mu(I'\setminus J_{1,j})>\dfrac{(2t_j-1)(t_j-1)}{t_j^4}$ for all $j\in[n]$. This contradicts Lemma \ref{Lem5}, so \eqref{Eq19} must fail for some $j\in[n]$. Lemma \ref{Lem2} tells us that $c_j(I)$ is an independent set, and it is sorted because $I$ is sorted. By Lemma \ref{Lem4}, 
\begin{equation}\label{Eq23}
\mu(c_j(I)\setminus J_{1,j})<\frac{t_j-1}{t_j^5}.
\end{equation} 

We have 
\begin{equation}\label{Eq20}
\frac{2t_j-1}{t_j^4}<\frac{1}{t_j-1}\mu(I\setminus J_{1,j})\leq\mu(I\cap J_{2,j})\leq\mu(I)/2,
\end{equation}
where the last two inequalities follow from the fact that $I$ is sorted. We will prove that 
\begin{equation}\label{Eq21}
\mu(\partial(I\cap J_{2,j})\cap J_{1,j}\cap c_j(I))\leq\mu(I\cap J_{2,j})
\end{equation}
by constructing an injection $\psi:\partial(I\cap J_{2,j})\cap J_{1,j}\cap c_j(I)\to I\cap J_{2,j}$. If $w\in\partial(I\cap J_{2,j})\cap J_{1,j}\cap c_j(I)$, then $w_j=1$. Let $\psi(w)=(w_1,\ldots,w_{j-1},2,w_{j+1},\ldots,w_n)$. This map is clearly injective, so we just need to check that $\psi(w)$ is actually an element of $I\cap J_{2,j}$. We know that $\psi(w)\in J_{2,j}$, so we must check that $\psi(w)\in I$. Because $w\in c_j(I)$, there exists $z\in I$ such that $z_i=w_i$ for all $i\neq j$. Since $w\in\partial (I\cap J_{2,j})$, there exists $y\in I$ such that $y$ is adjacent to $w$ and $y_j=2$. This means that $y_i\neq w_i=z_i$ for all $i\neq j$. Because $y$ and $z$ are distinct elements of the independent set $I$, they are not adjacent. This means that they must agree in some coordinate, which must be the $j^\text{th}$ coordinate. It follows that $z_j=2$, so $\psi(w)=z\in I$ as desired. 

Using \eqref{Eq21} and Corollary \ref{Cor1}, we find that \[\mu(J_{1,j}\setminus c_j(I))\geq\mu((\partial(I\cap J_{2,j})\cap J_{1,j})\setminus c_j(I))=\mu(\partial(I\cap J_{2,j})\cap J_{1,j})-\mu((\partial(I\cap J_{2,j})\cap J_{1,j})\cap c_j(I))\] \[\geq\mu(\partial(I\cap J_{2,j})\cap J_{1,j})-\mu(I\cap J_{2,j})=\frac{1}{t_j-1}\mu(\partial(I\cap J_{2,j}))-\mu(I\cap J_{2,j})\]
\begin{equation}\label{Eq22} \geq\frac{1}{t_j-1}\Phi(G,\mu(I\cap J_{2,j}))-\mu(I\cap J_{2,j})\geq\frac{1}{t_j-1}\mu(I\cap J_{2,j})^{1/\eta(t_n)}-\mu(I\cap J_{2,j}).
\end{equation}

Finally, combining \eqref{Eq23} and \eqref{Eq22} gives \[\mu(I)=\mu(c_j(I))=\mu(c_j(I)\cap J_{1,j})+\mu(c_j(I)\setminus J_{1,j})=\frac{1}{t_j}-\mu(J_{1,j}\setminus c_j(I))+\mu(c_j(I)\setminus J_{1,j})\] \[<\frac{1}{t_j}-\mu(J_{1,j}\setminus c_j(I))+\frac{t_j-1}{t_j^5}\leq\frac{1}{t_j}+\frac{t_j-1}{t_j^5}-\frac{1}{t_j-1}\mu(I\cap J_{2,j})^{1/\eta(t_n)}+\mu(I\cap J_{2,j})\] \[<\frac{1}{t_j}+\frac{1}{t_j^4}-\frac{1}{t_j-1}\mu(I\cap J_{2,j})^{1/\eta(t_n)}+\mu(I\cap J_{2,j}).\] However, this is a contradiction because we can apply Lemma \ref{Lem7} with $t=t_n$, $x=t_j$, $\nu=\mu(I\cap J_{2,j})$, and $m=\mu(I)$ (\eqref{Eq20} guarantees that the hypotheses of Lemma \ref{Lem7} are satisfied) to find that \[\frac{1}{t_j}+\frac{1}{t_j^4}-\frac{1}{t_j-1}\mu(I\cap J_{2,j})^{1/\eta(t_n)}+\mu(I\cap J_{2,j})\leq \mu(I). \qedhere\]
\end{proof}

We are finally prepared to complete the proof of Theorem \ref{Thm6}. A brief overview of the proof is as follows. We first assume the graph $G$ under consideration is a direct product of complete graphs. We use Lemma \ref{Lem6} to see that there is a choice of $j$ such that the quantity $\delta$ in Proposition \ref{Prop2} is somewhat small. Proposition \ref{Prop2} then allows us to deduce that $\delta$ is very small, which in turn allows us to prove the theorem in this case. To complete the proof, we show how to deduce the general form of the theorem from the version for direct products of complete graphs.  
 
\begin{proof}[Proof of Theorem \ref{Thm6}]
Let $G=\prod_{i=1}^nK[u_i,t_i]$ be as in the statement of the theorem, and assume for the moment that $u_i=1$ for all $i\in[n]$. In other words, $G=\prod_{i=1}^nK_{t_i}$, where $t_1\geq\cdots\geq t_n\geq 3$. Let $I\subseteq V(G)$ be an independent set with $\mu(I)=\dfrac{1}{t_n}(1-\varepsilon)>\omega(t_n)$. By relabeling the vertices in each of the graphs $K_{t_i}$ if necessary, we can assume $I$ is sorted. According to Lemma \ref{Lem6}, we can choose $j\in[n]$ such that $\mu(I\setminus J_{1,j})<\dfrac{t_j-1}{t_j^5}<\dfrac{1}{t_j^4}$. Let $\delta=\mu(I\setminus J_{1,j})$. We know from Proposition \ref{Prop2} that 
\begin{equation}\label{Eq29}\varepsilon\geq 1-\frac{t_n}{t_j}-\delta t_n+\frac{t_n}{t_j-1}\left(\frac{\delta}{t_j-1}\right)^{1/\eta(t_n)}.
\end{equation} We will first prove that \begin{equation}\label{Eq28}
-\delta t_n+\frac{t_n}{t_j-1}\left(\frac{\delta}{t_j-1}\right)^{1/\eta(t_n)}>0.
\end{equation}
It will then follow from \eqref{Eq29} that $\varepsilon>1-\dfrac{t_n}{t_j}$, which is equivalent to $t_j<\dfrac{t_n}{1-\varepsilon}$. 

Let \[f(x)=-t_nx+\frac{t_n}{t_j-1}\left(\frac{x}{t_j-1}\right)^{1/\eta(t_n)}.\] If $0<x<\dfrac{1}{t_j^4}$, then \[f'(x)=t_n\left(\frac{1}{(t_j-1)^{1+1/\eta(t_n)}}\frac{1}{\eta(t_n)}x^{1/\eta(t_n)-1}-1\right)>t_n\left(\frac{1}{t_j^{1+1/\eta(t_n)}}\frac{1}{\eta(t_n)}x^{1/\eta(t_n)-1}-1\right)\] \[>t_n\left(\frac{1}{t_j^{1+1/\eta(t_n)}}\frac{1}{\eta(t_n)}\left(\frac{1}{t_j^4}\right)^{1/\eta(t_n)-1}-1\right)=\frac{t_n}{\eta(t_n)}\left(t_j^{3-5/\eta(t_n)}-\eta(t_n)\right)\] \[\geq \frac{t_n}{\eta(t_n)}\left(t_n^{3-5/\eta(t_n)}-\eta(t_n)\right).\] It is easy to verify that this last expression is positive, so $f(x)$ is increasing for $0<x<\dfrac{1}{t_j^4}$. This immediately implies \eqref{Eq28} since $f(0)=0$. 

Next, let \[g(x)=1-\frac{t_n}{x}-\delta t_n+\frac{t_n}{x-1}\left(\frac{\delta}{x-1}\right)^{1/\eta(t_n)}.\] Suppose $t_n\leq x\leq t_j$. We have \[g'(x)=t_n\left(x^{-2}-\delta^{1/\eta(t_n)}\left(1+1/\eta(t_n)\right)(x-1)^{-2-1/\eta(t_n)}\right)\] \[=\frac{\delta^{1/\eta(t_n)}t_n}{(x-1)^{2+1/\eta(t_n)}}\left(\frac{(x-1)^{2+1/\eta(t_n)}}{x^2\delta^{1/\eta(t_n)}}-\left(1+1/\eta(t_n)\right)\right)\] \[\geq\frac{\delta^{1/\eta(t_n)}t_n}{(x-1)^{2+1/\eta(t_n)}}\left(\frac{(x-1)^{2+1/\eta(t_n)}}{x^2(1/t_j^4)^{1/\eta(t_n)}}-\left(1+1/\eta(t_n)\right)\right)\]\[\geq\frac{\delta^{1/\eta(t_n)}t_n}{(x-1)^{2+1/\eta(t_n)}}\left(\frac{(x-1)^{2+1/\eta(t_n)}}{x^2(1/x^4)^{1/\eta(t_n)}}-\left(1+1/\eta(t_n)\right)\right)\] 
\begin{equation}\label{Eq30}
\geq\frac{\delta^{1/\eta(t_n)}t_n}{(x-1)^{2+1/\eta(t_n)}}\left(\frac{(x-1)^{2+1/\eta(t_n)}}{x^{2-4/\eta(t_n)}}-\left(1+1/\eta(t_n)\right)\right).
\end{equation} It is easy to check that $\dfrac{(x-1)^{2+1/\eta(t_n)}}{x^{2-4/\eta(t_n)}}$ is increasing in $x$, so \[\frac{(x-1)^{2+1/\eta(t_n)}}{x^{2-4/\eta(t_n)}}-\left(1+1/\eta(t_n)\right)\geq\frac{(t_n-1)^{2+1/\eta(t_n)}}{t_n^{2-4/\eta(t_n)}}-\left(1+1/\eta(t_n)\right)\] \[=(t_n-1)^{1/\eta(t_n)}\frac{(t_n-1)^2}{t_n^2}t_n^{4/\eta(t_n)}-\left(1+1/\eta(t_n)\right)=(t_n-1)^{1/\eta(t_n)}\frac{t_n^2}{(t_n-1)^2}-\left(1+1/\eta(t_n)\right)\] \[>\frac{t_n}{t_n-1}-\left(1+1/\eta(t_n)\right)=\frac{1}{t_n-1}-\frac{1}{\eta(t_n)}>0.\] Combining this with \eqref{Eq30} shows that $g(x)$ is increasing in $x$ when $t_n\leq x\leq t_j$. In particular, $g(t_j)\geq g(t_n)$. Referring back to \eqref{Eq29}, we see that 
\begin{equation}\label{Eq31}
\varepsilon\geq -\delta t_n+\frac{t_n}{t_n-1}\left(\frac{\delta}{t_n-1}\right)^{1/\eta(t_n)}.
\end{equation} At this point, we simply cite the proof of Lemma 11 in \cite{Brakensiek}. In that proof, Brakensiek obtains the equation \eqref{Eq31} under the weaker assumption that $\delta<\dfrac{1}{t_n^3}$ and proves that 
\begin{equation}\label{Eq43}
\delta<4\varepsilon^{\eta(t_n)}
\end{equation} (note that he uses the symbol $t$ in place of $t_n$). Applying the exact same argument proves that \eqref{Eq43} holds in our case as well. 

We now prove the theorem in the more general case in which $G=\prod_{i=1}^nK[u_i,t_i]$ with $t_1\geq\cdots\geq t_n\geq 3$. Let $G'=\prod_{i=1}^nK_{t_i}$, and consider the collapsing map $\coll:V(G)\to V(G')$ defined in Section \ref{Iso}. More precisely, $\coll=\coll_1\times\cdots\times \coll_n$, where $\coll_i:V(K[u_i,t_i])\to V(K_{t_i})$ sends every vertex in the partite set $X_i(a)$ to the vertex $a$. Because the complete multipartite graphs $K[u_i,t_i]$ in the product defining $G$ are \emph{balanced}, we have 
\begin{equation}\label{Eq44}
\mu(\coll^{-1}(T))=\mu(T)
\end{equation} for all $T\subseteq G'$. We use $J_{a,i}$ to refer to the subset \[V(K[u_1,t_1])\times\cdots\times V(K[u_{i-1},t_{i-1}]\times X_i(a)\times V(K[u_{i+1},t_{i+1}])\times\cdots\times V(K[u_{i+1},t_n])\] of $V(G)$ and use $J_{a,i}'$ to refer to the subset \[[t_1]\times\cdots\times[t_{i-1}]\times\{a\}\times[t_{i+1}]\times\cdots\times[t_n]=\coll(J_{a,i})\] of $V(G')$. Let $I\subseteq V(G)$ be an independent set with $\mu(I)=\dfrac{1}{t_n}(1-\varepsilon)>\omega(t_n)$. The collapsing map sends independent sets to independent sets and satisfies $\mu(\coll(T))=\mu(\coll^{-1}(\coll(T)))\geq\mu(T)$ for all $T\subseteq V(G)$. Therefore, the set $I'=\coll(I)$ is an independent set of $G'$ with $\mu(I')=\dfrac{1}{t_n}(1-\varepsilon')$ for some $\varepsilon'\leq\varepsilon$. We already know the theorem holds for direct products of complete graphs (such as $G'$), so there exist $j\in[n]$ and $a\in[t_j]$ such that \[t_j<\frac{t_n}{1-\varepsilon'}\quad\text{and}\quad\mu(I'\setminus J_{a,j}')<4(\varepsilon')^{\eta(t_n)}.\] We have $I\setminus J_{a,j}\subseteq\coll^{-1}(I'\setminus J_{a,j}')$, so \[\mu(I\setminus J_{a,j})\leq\mu(\coll^{-1}(I'\setminus J_{a,j}'))=\mu(I'\setminus J_{a,j}')<4(\varepsilon')^{\eta(t_n)}\leq 4\varepsilon^{\eta(t_n)}\] as desired. 
\end{proof}

\section{Concluding Remarks}
We have proven Burcroff's conjecture that $\IR(G)=\alpha(G)$ whenever $G$ is a direct product of balanced complete multipartite graphs except in $37$ exceptional cases. Our proof relied on the fact that the complete multipartite graphs in the product are balanced. As mentioned in the introduction, our new Conjecture \ref{Conj2} strengthens Burcroff's conjecture by removing the assumption that the graphs in the product are balanced.

In Theorem \ref{Thm2}, we gave an explicit recursive formula for the vertex isoperimetric profile of the graph $\prod_{i=1}^nH_i$ when $H_1,\ldots,H_n$ are complete multipartite graphs satisfying $\beta(H_1)\leq\cdots\leq\beta(H_n)$ and \[\prod_{k\in A}\frac{1-\beta(H_k)}{\beta(H_k)}\geq\frac{1-\beta(H_n)}{\beta(H_n)}\] for all nonempty $A\subseteq [n-1]$. This last condition was satisfied for all the graphs we considered in our applications, but it would still be interesting to compute the vertex isoperimetric profiles of direct products of complete multipartite graphs that fail to satisfy this condition. It would also be interesting to prove an independent set stability result like Theorem \ref{Thm6} for direct products of complete multipartite graphs that are not necessarily balanced. 

It could be possible to weaken the hypothesis that $\mu(I)>\omega(t_n)$ in Theorem \ref{Thm6}. Doing so could allow one to prove Burcroff's conjecture for several of the $37$ remaining cases. Alternatively, one could see if an argument similar to the one used in \cite{defantandiyer} to prove the conjecture in the case $n\leq 3$ could also handle the case $n=4$; this would prove $26$ of the remaining $37$ cases.

\section{Acknowledgments}
The first author was supported in part by an ISF Grant No. 281/17 and by the Simons Foundation. The second author was supported by a Fannie and John Hertz Foundation Fellowship and an NSF Graduate Research Fellowship.

\end{document}